\documentclass[11pt]{article}

\usepackage[margin=1.2in]{geometry}

\usepackage{amsmath}
\usepackage{amsthm}
\usepackage{graphicx}
\usepackage{color}
\usepackage{amsfonts}
\usepackage{amssymb}
\usepackage{mathrsfs}
\usepackage{amscd}
\usepackage{url}

\newcommand{\re}{\mathbb{R}}

\newcommand{\half}{\frac{1}{2}}
\newcommand{\lmd}{\lambda}

\newcommand{\eps}{\epsilon}

\def\af{\alpha}
\def\bt{\beta}

\newcommand{\reff}[1]{(\ref{#1})}

\newcommand{\mc}[1]{\mathcal{#1}}

\newcommand{\supp}[1]{\mbox{supp}(#1)}

\newcommand{\cv}[1]{\mbox{conv}(#1)}

\newcommand{\bdes}{\begin{description}}
\newcommand{\edes}{\end{description}}

\newcommand{\bal}{\begin{align}}
\newcommand{\eal}{\end{align}}

\newcommand{\bnum}{\begin{enumerate}}
\newcommand{\enum}{\end{enumerate}}

\newcommand{\bit}{\begin{itemize}}
\newcommand{\eit}{\end{itemize}}

\newcommand{\bea}{\begin{eqnarray}}
\newcommand{\eea}{\end{eqnarray}}
\newcommand{\be}{\begin{equation}}
\newcommand{\ee}{\end{equation}}

\newcommand{\baray}{\begin{array}}
\newcommand{\earay}{\end{array}}

\newcommand{\bsry}{\begin{subarray}}
\newcommand{\esry}{\end{subarray}}

\newcommand{\bca}{\begin{cases}}
\newcommand{\eca}{\end{cases}}

\newcommand{\bcen}{\begin{center}}
\newcommand{\ecen}{\end{center}}

\newcommand{\bbm}{\begin{bmatrix}}
\newcommand{\ebm}{\end{bmatrix}}

\newcommand{\bmx}{\begin{matrix}}
\newcommand{\emx}{\end{matrix}}

\newcommand{\bpm}{\begin{pmatrix}}
\newcommand{\epm}{\end{pmatrix}}

\newcommand{\btab}{\begin{tabular}}
\newcommand{\etab}{\end{tabular}}

\newtheorem{theorem}{Theorem}[section]
\newtheorem{pro}[theorem]{Proposition}

\theoremstyle{definition}

\newtheorem{exm}[theorem]{Example}

\begin{document}

\title{ Convex Hulls of Quadratically Parameterized Sets
With Quadratic Constraints}

\author{Jiawang Nie\footnote{Department of Mathematics,
University of California, 9500 Gilman Drive, La Jolla, CA 92093.
Email: njw@math.ucsd.edu. The research was partially supported by NSF grants
DMS-0757212 and DMS-0844775.}
}

\date{October 10, 2011}

\maketitle

\centerline{\it \large Dedicated to Bill Helton on the occasion of his 65th birthday.}

\begin{abstract}
Let $V$ be a semialgebraic set parameterized as
\[
\{ (f_1(x), \ldots, f_m(x)): \, x \in T \}
\]
for quadratic polynomials $f_0, \ldots, f_m$ and a subset $T$ of $\re^n$.
This paper studies semidefinite representation of
the convex hull $\cv{V}$ or its closure, i.e.,
describing $\cv{V}$ by projections of spectrahedra
(defined by linear matrix inequalities).
When $T$ is defined by a single quadratic constraint,
we prove that $\cv{V}$ is equal to the first order moment type
semidefinite relaxation of $V$, up to taking closures.
Similar results hold when every $f_i$ is a quadratic form
and $T$ is defined by two homogeneous (modulo constants) quadratic constraints,
or when all $f_i$ are quadratic rational functions
with a common denominator and $T$ is defined by a single quadratic constraint,
under some general conditions.
\end{abstract}

\section{Introduction}

A basic question in convex algebraic geometry is to find
convex hulls of semialgebraic sets.
A typical class of semialgebraic sets is parameterized by multivariate polynomial
functions defined on some sets.
Let $V\subset \re^m$ be a set parameterized as
\be
V = \{ (f_1(x), \ldots, f_m(x)): \, x \in T \}
\ee
with every $f_i(x)$ being a polynomial and
$T$ a semialgebraic set in $\re^n$.
We are interested in finding a representation for
the convex hull $\cv{V}$ of $V$ or its closure,
based on $f_1,\ldots, f_m$ and $T$.
Since $V$ is semialgebraic, $\cv{V}$ is a convex semialgebraic set.
Thus, one wonders whether $\cv{V}$ is representable
by a spectrahedron or its projection, i.e.,
as a feasible set of {\it semidefinite programming (SDP)}.
A {\it spectrahedron} of $\re^k$ is a set defined
by a linear matrix inequality (LMI) like
\[
L_0 + w_1 L_1 + \cdots + w_k L_k \succeq 0
\]
for some constant symmetric matrices $L_0,\ldots,L_k$.
Here the notation $X \succeq 0$ (resp. $X \succ 0$)  means the matrix
$X$ is positive semidefinite (resp. definite).
Equivalently, a spectrahedron is the intersection
of a positive semidefinite cone and an affine linear subspace.
Not every convex semialgebraic set is a spectrahedron,
as found by Helton and Vinnikov \cite{HV}.
Actually, they \cite{HV} proved a necessary condition called {\it rigid convexity}
for a set to be a spectrahedron.
They also proved that rigid convexity is sufficient in the two dimensional case.
Typically, projections of spectrahedra are required in representing convex sets
(if so, they are also called {\it semidefinite representations}).
It has been found that a very general class of convex sets
are representable as projections of spectrahedra, as shown in \cite{HN1,HN2}.
The proofs used sum of squares (SOS) type representations of
polynomials that are positive on compact semialgebraic sets,
as given by Putinar \cite{Putinar} or Schm\"{u}dgen \cite{Smg}.
More recent work about semidefinite representations of convex semialgebraic sets
can be found in \cite{HN08,Las06,Las08,Nie08,Nie-MOR}.

A natural semidefinite relaxation for the convex hull $\cv{V}$
can be obtained by using the moment approach \cite{Las06,Par06}.
To describe it briefly, we consider the simple case that
$n=1$, $T=\re$ and $(f_1(x),f_2(x), f_3(x)) = (x^2,x^3,x^4)$ with $m=3$.
The most basic moment type semidefinite relaxation of $\cv{V}$ in this case is
\[
R = \left\{ (y_2, y_3, y_4):
\bbm  1 & y_1 & y_2 \\ y_2 & y_2 & y_3 \\ y_2 & y_3 & y_4 \ebm \succeq 0
\mbox{ for some } y_1 \in \re
\right\}.
\]
The underlying idea is to replace each monomial $x^i$ by
a lifting variable $y_i$ and to pose the LMI in the definition of $R$,
which is due to the fact that
\[
\bbm 1 \\ x \\ x^2 \ebm \bbm 1 \\ x \\ x^2 \ebm^T \quad = \quad
\bbm  1 & x & x^2 \\ x & x^2 & x^3 \\ x^2 & x^3 & x^4 \ebm
\quad \succeq \quad 0 \qquad \forall \,\, x\in \re.
\]
If $n=1$, the sets $R$ and $\cv{V}$ (or their closures) are equal
(cf. \cite{Par06}). When $T =\re^n$ with $n>1$,
we have similar results if every $f_i$ is quadratic
or every  $f_i$ is quartic but $n=2$ (cf. \cite{Hen09}). However,
in more general cases, similar results typically do not exist anymore.

In this paper, we consider the special case that every $f_i$ is quadratic
and $T$ is a quadratic set of $\re^n$.
When $T$ is defined by a single quadratic constraint,
we will show that the first order moment type semidefinite relaxation
represents $\cv{V}$ or its closure as the projection of a spectrahedron (Section~2).
This is also true when every $f_i$ is a quadratic form
and $T$ is defined by two
homogeneous (modulo constants) quadratic constraints (Section~3),
or when all $f_i$ are quadratic rational functions
with a common denominator and $T$ is defined by a single
quadratic constraint (Section~4), under some general conditions.

\bigskip
{\it Notations} \,
The symbol $\re$ (resp. $\re_+$) denotes the set of (resp. nonnegative) real numbers.
For a symmetric matrix, $X\prec 0$ means $X$ is negative definite ($-X\succ 0$);
$\bullet$ denotes the standard Frobenius inner product in matrix spaces;
$\|\cdot \|_2$ denotes the standard $2$-norm.
The superscript $^T$ denotes the transpose of a matrix;
$\overline{K}$ denotes the closure of a set $K$ in a Euclidean space,
and $\cv{K}$ denotes the convex hull of $K$.
Given a function $q(x)$ defined on $\re^n$, denote
\[
S(q)=\{x\in \re^n:\, q(x)\geq 0\}, \quad E(q)=\{x\in \re^n:\, q(x) = 0\}.
\]

\section{A single quadratic constraint}
\setcounter{equation}{0}

Suppose $V\subset \re^m$ is a semialgebraic set parameterized as
\be \label{q-var:one-q}
V = \{ (f_1(x), \ldots, f_m(x)): \, x \in T \}
\ee
where every $f_i(x) = a_i + b_i^Tx+ x^TF_ix$ is quadratic and
$T \subseteq \re^n$ is defined by a single
quadratic inequality $q(x)\geq0$ or equality $q(x)=0$.
The $a_i,b_i, F_i$ are vectors
or symmetric matrices of proper dimensions. Similarly, write
\[
q(x) = c + d^Tx + x^TQx.
\]
For every $x\in T$, it always holds that for $X=xx^T$
\[
f_i(x) = a_i + b_i^Tx+F_i\bullet X, \quad
q(x) = c + d^Tx +Q \bullet X \geq 0, \quad
\bbm 1 & x^T \\ x &  X \ebm \succeq 0.
\]
Clearly, when $T=S(q)$, the convex hull $\cv{V}$ of $V$ is contained in the convex set
\[
\mc{W}_{in} =  \left\{
(a_1 + b_1^Tx+F_1\bullet X, \ldots, a_m + b_m^Tx+F_m\bullet X )
\left| \baray{c} \bbm 1 & x^T \\ x &  X \ebm \succeq 0, \\
c + d^Tx +Q \bullet X \geq 0 \earay \right.
\right\}.
\]
When $T=E(q)$, the convex hull $\cv{V}$ is then contained in the convex set
\[
\mc{W}_{eq}  =  \left\{
(a_1 + b_1^Tx+F_1\bullet X, \ldots, a_m + b_m^Tx+F_m\bullet X )
\left| \baray{c} \bbm 1 & x^T \\ x &  X \ebm \succeq 0, \\
 c + d^Tx +Q \bullet X  = 0 \earay \right.
\right\}.
\]
Both $\mc{W}_{in}$ and $\mc{W}_{eq}$ are projections of spectrahedra.
One wonders whether $\mc{W}_{in}$ or $\mc{W}_{eq}$ is equal to $\cv{V}$.
Interestingly, this is almost always true, as given below.

\begin{theorem} \label{thm:S>=0}
Let $V,T,W_{in},W_{eq},q$ be defined as above, and $T \ne \emptyset $.
\bit
\item [(i)] Let $T=S(q)$. If $T$ is compact, then $\cv{V} = \mc{W}_{in}$;
otherwise,  $\overline{\cv{V}} =  \overline{\mc{W}_{in}} $.

\item [(ii)] Let $T=E(q)$. If $T$ is compact, then $\cv{V} = \mc{W}_{eq}$;
otherwise,  $\overline{\cv{V}} =  \overline{\mc{W}_{eq}} $.

\eit
\end{theorem}

To prove the above theorem, we need a result on quadratic moment problems.
A {\it quadratic moment sequence} is a triple $(t,z, Z) \in \re \times \re^n \times \re^{n\times n}$
with $Z$ symmetric. We say $(t,z,Z)$ admits a representing measure supported on $T$
if there exists a positive Borel measure $\mu$ with its support $\supp{\mu} \subseteq T$ and
\[
t = \int 1 \,d \mu, \quad z = \int x d \mu, \quad Z = \int xx^T d \mu.
\]
Denote by $\mathscr{R}(T)$ the set of all such quadratic moment sequences
$(t,z,Z)$ satisfying the above.

\begin{theorem}(\cite[Theorems~4.7,4.8]{FiNi01})  \label{thm:FiNi}
Let $q(x) = c + d^Tx + x^TQx$, $T=S(q)$ or $E(q)$ be nonempty,
and $(t,z,Z)$ be a quadratic moment sequence satisfying
\[
 \bbm 1 & z^T \\ z &  Z \ebm \succeq 0, \quad
\bca
c + d^Tz +Q \bullet Z \geq 0, & \text{ if } T=S(q); \\
c + d^Tz +Q \bullet Z = 0, & \text{ if } T=E(q).
\eca
\]
\bit
\item [(i)] If $S(q)$ is compact, then  $(t,z,Z) \in \mathscr{R}(S(q))$;
otherwise,  $(t,z,Z) \in \overline{\mathscr{R}(S(q))}$.

\item [(ii)] If $E(q)$ is compact, then  $(t,z,Z) \in \mathscr{R}(E(q))$;
otherwise,  $(t,z,Z) \in \overline{\mathscr{R}(E(q))}$.
\eit

\end{theorem}

\bigskip
\noindent
{\it Proof of Theorem~\ref{thm:S>=0}} \,
(i) We have already seen that $\cv{V} \subseteq \mc{W}_{in}$,
which clearly implies $\overline{\cv{V}} \subseteq \overline{\mc{W}_{in}}$.
Suppose $(x,X)$ is a pair satisfying the conditions in $\mc{W}_{in}$.

If $T=S(q)$ is compact, by Theorem~\ref{thm:FiNi},
the quadratic moment sequence $(1,x,X)$ admits a representing measure supported in $T$.
By the Bayer-Teichmann Theorem \cite{BT},
the triple $(1,x,X)$ also admits a measure
having a finite support contained in $T$.
So, there exist $u_1,\ldots, u_r \in T$
and scalars $\lmd_1>0,\ldots, \lmd_r>0$ such that
\[
\bbm 1 & x^T \\ x  & X \ebm =
\lmd_1 \bbm 1 & u_1^T \\ u_1  & u_1u_1^T \ebm + \cdots +
\lmd_r \bbm 1 & u_r^T \\ u_r  & u_ru_r^T \ebm.
\]
The above implies that
\[
(a_1 + b_1^Tx+F_1\bullet X, \ldots, a_m + b_m^Tx+F_m\bullet X ) =
\sum_{i=1}^r \lmd_i (f_1(u_i), \ldots, f_m(u_i)).
\]
Clearly, $\lmd_1+\cdots+\lmd_r=1$. So, $\mc{W}_{in} \subseteq \cv{V}$
and hence $\mc{W}_{in} = \cv{V}$.

If $T=S(q)$ is noncompact,
the quadratic moment sequence $(1,x,X) \in \overline{\mathscr{R}(T)}$, and
\[
(1,x,X) = \lim_{k\to\infty} \, (1, x^{(k)}, X^{(k)}), \, \mbox{ with every } \,
 (1, x^{(k)}, X^{(k)}) \in \mathscr{R}(T).
\]
As we have seen in (i), every
\[
(a_1 + b_1^Tx^{(k)}+F_1\bullet X^{(k)}, \ldots, a_m + b_m^Tx^{(k)}+F_m\bullet X^{(k)} )
\in  \cv{V}.
\]
This implies
\[
(a_1 + b_1^Tx+F_1\bullet X, \ldots, a_m + b_m^Tx+F_m\bullet X )
\in \overline{\cv{V}}.
\]
So, $\overline{\mc{W}_{in}} \subseteq \overline{\cv{V}}$
and consequently $\overline{\mc{W}_{in}} = \overline{\cv{V}}$.

(ii) can be proved in the same way as for (i).
\qed

\begin{figure}[htb]
\centering
\btab{c}
\includegraphics[height=.7\textwidth]{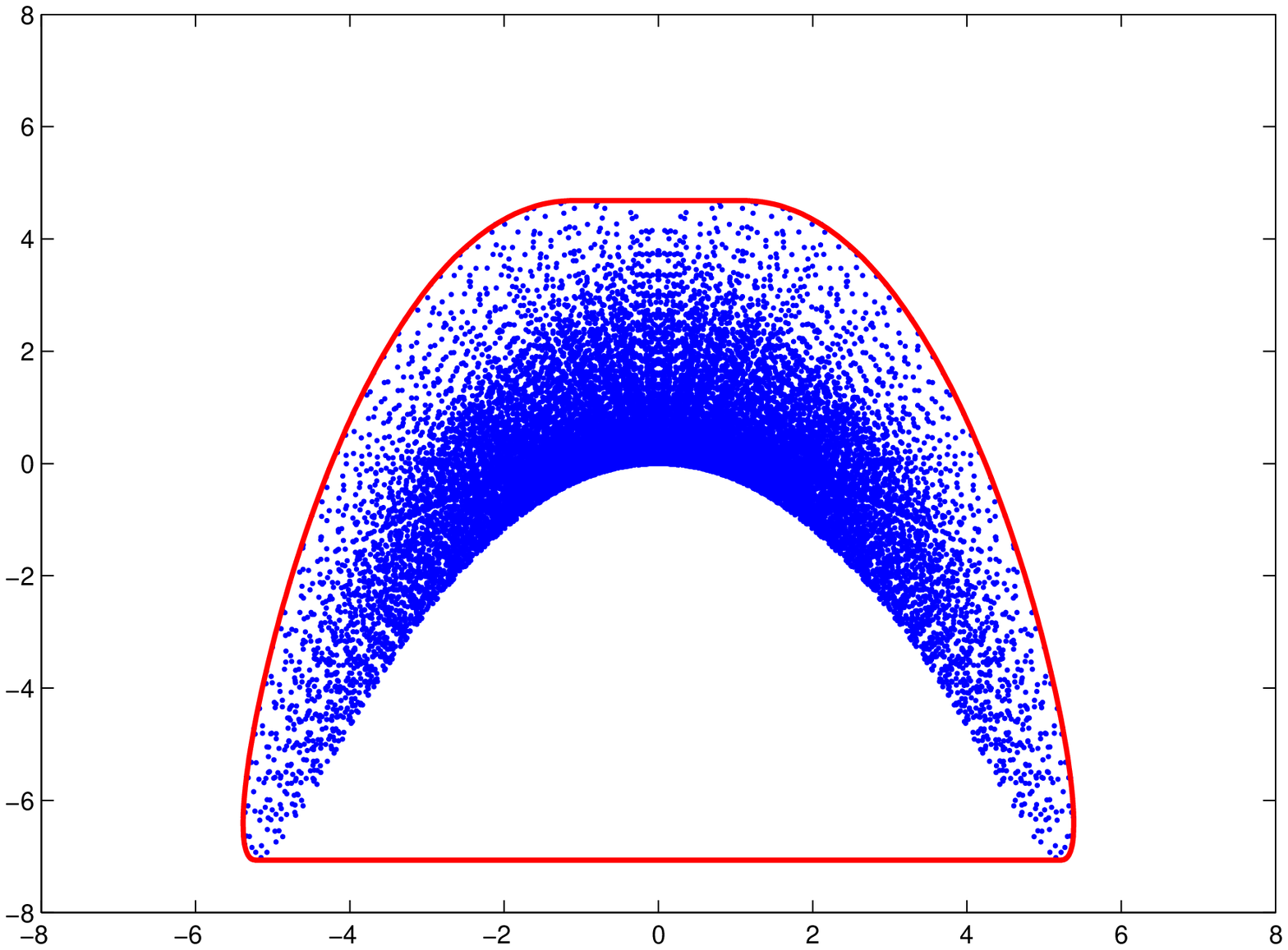}
\etab
\caption{The dotted area is the set $V$ in Example~\ref{emp:Qpar:ball},
and the outer curve is the boundary of the convex hull $\cv{V}$. }
\label{fig-Qpar:Ball}
\end{figure}

\begin{exm} \label{emp:Qpar:ball}
Consider the parametrization
\[
V = \left\{ \left(3x_1-2x_2-4x_3,5x_1x_2+7x_1x_3-9x_2x_3 \right): \,
\|x\|_2 \leq 1 \right \}.
\]
The set $V$ is drawn in the dotted area of Figure~\ref{fig-Qpar:Ball}.
By Theorem~\ref{thm:S>=0},
the convex hull $\cv{V}$ is given by the semidefinite representation
\[
\left\{ \bpm 3x_1-2x_2-4x_3 \\  5X_{12}+7X_{13}-9X_{23} \epm
\left| \baray{c}
\bbm 1 & x_1 & x_2 & x_3 \\ x_1 & X_{11} & X_{12} & X_{13} \\
x_2 & X_{12} & X_{22} & X_{23} \\ x_3 & X_{13} & X_{23} & X_{33} \ebm \succeq 0, \\
1- X_{11} - X_{22} - X_{33} \geq 0
\earay \right.
\right\}.
\]
The boundary of the above set is the outer curve
in Figure~\ref{fig-Qpar:Ball}. One can easily see that $\cv{V}$
is correctly given by the above semidefinite representation.
\qed

\end{exm}

\section{Two homogeneous constraints}
\setcounter{equation}{0}

Suppose $V\subset \re^m$ is a semialgebraic set parameterized as
\be
V = \{ (x^TA_1x, \ldots, x^TA_mx): \, x \in T \}.
\ee
Here, every $A_i$ is a symmetric matrix and $T$
is defined by two homogeneous (modulo constants)
inequalities/equalities $h_j(x)\geq 0$ or $h_j(x)=0$, $j=1,2$.
Write
\[
h_1(x) = x^TB_1x - c_1,  \quad h_2(x) = x^TB_2 x - c_2,
\]
for symmetric matrices $B_1,B_2$. The set $T$ is one of the four cases:
\[
E(h_1) \cap E(h_2), \quad S(h_1) \cap E(h_2), \quad
E(h_1) \cap S(h_2), \quad S(h_1) \cap S(h_2).
\]
Note the relations:
\[
x^TA_ix = A_i \bullet (xx^T) \quad \, (1\leq i\leq m), \quad xx^T \succeq 0,
\]
\[
x^TB_1x = B_1 \bullet (xx^T), \quad x^TB_2x = B_2 \bullet (xx^T).
\]
If we replace $xx^T$ by a symmetric matrix $X\succeq 0$,
then $V$, as well as $\cv{V}$,
is contained respectively in the following projections of spectrahedra:
\begin{align}
\baray{rcl}
\mc{H}_{e,e} &=& \{ (A_1 \bullet X, \ldots, A_m \bullet X): \,
X\succeq 0, \, B_1\bullet X = c_1, B_2 \bullet X = c_2 \}, \\
\mc{H}_{i,e} &=& \{ (A_1 \bullet X, \ldots, A_m \bullet X): \,
X\succeq 0, \, B_1\bullet X \geq c_1, B_2 \bullet X = c_2 \}, \\
\mc{H}_{e,i} &=& \{ (A_1 \bullet X, \ldots, A_m \bullet X): \,
X\succeq 0, \, B_1\bullet X = c_1, B_2 \bullet X \geq c_2 \}, \\
\mc{H}_{i,i} &=& \{ (A_1 \bullet X, \ldots, A_m \bullet X): \,
X\succeq 0, \, B_1\bullet X \geq c_1, B_2 \bullet X \geq c_2 \}.
\earay
\end{align}
To analyze whether they represent $\cv{V}$ respectively,
we need the following conditions for the four cases:

\be
\bca
C_{e,e}: \, \exists (\mu_1,\mu_2)\in \re \times \re, \, s.t. \quad
\mu_1 B_1 + \mu_2 B_2 \prec 0, \\
C_{i,e}: \, \exists (\mu_1,\mu_2)\in \re_+ \times \re, \, s.t. \quad
\mu_1 B_1 + \mu_2 B_2 \prec 0, \\
C_{e,i}: \, \exists (\mu_1,\mu_2)\in \re \times \re_+, \, s.t. \quad
\mu_1 B_1 + \mu_2 B_2 \prec 0, \\
C_{i,i}: \, \exists (\mu_1,\mu_2)\in \re_+ \times \re_+, \, s.t. \quad
\mu_1 B_1 + \mu_2 B_2 \prec 0.
\eca
\ee

\begin{theorem} \label{hull:2hmg}
Let $V\ne\emptyset,\mc{H}_{e,e}, \mc{H}_{i,e}, \mc{H}_{e,i}, \mc{H}_{i,i}$ be defined as above.
Then we have
\be
\cv{V}=
\bca
\mc{H}_{e,e}, & \text{ if } T = E(h_1) \cap E(h_2) \mbox{ and } C_{e,e} \text{ holds;} \\
\mc{H}_{i,e}, & \text{ if } T = S(h_1) \cap E(h_2) \mbox{ and } C_{i,e} \text{ holds;} \\
\mc{H}_{e,i}, & \text{ if } T = E(h_1) \cap S(h_2) \mbox{ and } C_{e,i} \text{ holds;} \\
\mc{H}_{i,i}, & \text{ if } T = S(h_1) \cap S(h_2) \mbox{ and } C_{i,i} \text{ holds.}
\eca
\ee

\end{theorem}

\begin{proof}
We just prove for the case that $T=S(h_1)\cap S(h_2)$
and condition $C_{i,i}$ holds.
The proof is similar for the other three cases.
The condition $C_{i,i}$ implies that for some $\mu_1 \geq 0, \mu_2 \geq 0, \eps >0$
\[
-\mu_1 c_1 - \mu_2 c_2 \geq  x^T(-\mu_1 B_1 - \mu_2 B_2)x \geq \eps \|x\|_2^2.
\]
So, $T$ and $\cv{V}$ are compact.
Clearly, $\cv{V} \subseteq \mc{H}_{i,i}$.
We need to show $\mc{H}_{i,i} \subseteq \cv{V}$. 
Suppose otherwise it is false, then there exists a symmetric matrix $Z$ satisfying
\[
(A_1\bullet Z, \ldots, A_m \bullet Z) \not\in \cv{V},  \quad
B_1 \bullet Z \geq c_1, \quad B_2 \bullet Z \geq c_2, \quad Z \succeq 0.
\]
Because $\cv{V}$ is a closed convex set,
by the Hahn-Banach theorem, there exists a vector
$(\ell_0, \ell_1, \ldots, \ell_m) \ne 0$ satisfying
\[
\baray{c}
\ell_1 x^TA_1x + \cdots + \ell_m x^TA_mx \geq \ell_0 \quad \forall x \in T, \\
\ell_1 A_1 \bullet Z + \cdots + \ell_m  A_m \bullet Z < \ell_0.
\earay
\]
Consider the SDP problem
\be \label{P-sdp:ell}
\baray{rl}
p^*:=\min & \ell_1 A_1 \bullet X + \cdots + \ell_m  A_m \bullet X \\
s.t. &  X \succeq 0, \, B_1\bullet X \geq c_1, \, B_2 \bullet X \geq c_2.
\earay
\ee
Its dual optimization problem is
\be \label{D-sdp:ell}
\baray{rl}
\max & c_1 \lmd_1 + c_2 \lmd_2 \\
s.t. &  \sum_i \ell_i A_i - \lmd_1 B_1 - \lmd_2 B_2 \succeq 0, \, \lmd_1 \geq 0, \lmd_2 \geq 0.
\earay
\ee
The condition $C_{i,i}$ implies that the dual problem \reff{D-sdp:ell}
has nonempty interior. So, the primal problem \reff{P-sdp:ell} has an optimizer.
Define $\tilde{A}_0, \tilde{B}_1, \tilde{B}_2$ and a new variable $Y$ as:
\[
\tilde{A}_0 = \bbm  \sum_{i=1}^m \ell_iA_i & 0 & 0 \\ 0 & 0 & 0 \\ 0 & 0 & 0 \ebm,
\tilde{B}_1 = \bbm  B_1 & 0 & 0  \\ 0  & -1 & 0 \\ 0 & 0 & 0\ebm,
\tilde{B}_2 = \bbm  B_2 & 0 & 0  \\ 0  & 0 & 0 \\ 0 & 0 & -1\ebm,
Y = \bbm  X & Y_{12} \\  Y_{12}^T  & Y_{22} \ebm.
\]
They are all $(n+2) \times (n+2)$ symmetric matrices.
Clearly, the primal problem \reff{P-sdp:ell} is equivalent to
\be \label{P-sdp:tilde}
\baray{rl}
p^*:=\min & \tilde{A}_0 \bullet Y \\
s.t. &  Y \succeq 0, \, \tilde{B}_1 \bullet Y = c_1, \, \tilde{B}_1  \bullet Y = c_2.
\earay
\ee
It must also have an optimizer.
By Theorem~2.1 of Pataki \cite{Patk}, \reff{P-sdp:tilde}
has an extremal solution $U$ of rank r satisfying
\[
\half r(r+1) \leq 2.
\]
So, we must have $r=1$ and can write $Y=vv^T$.
Let $u=v(1:n)$. Then $u\in T$ and
\[
p^* = \ell_1 u^TA_1 u + \cdots + \ell_m u^TA_m u \geq \ell_0.
\]
However, $Z$ is also a feasible solution of \reff{P-sdp:ell},
and we get the contradiction
\[
p^* \leq \ell_1 A_1 \bullet Z + \cdots + \ell_m  A_m \bullet Z < p^*.
\]
Therefore, $\mc{H}_{i,i} \subseteq \cv{V}$ and they must be equal.
\end{proof}

\begin{figure}[htb]
\centering
\btab{c}
\includegraphics[height=.7\textwidth]{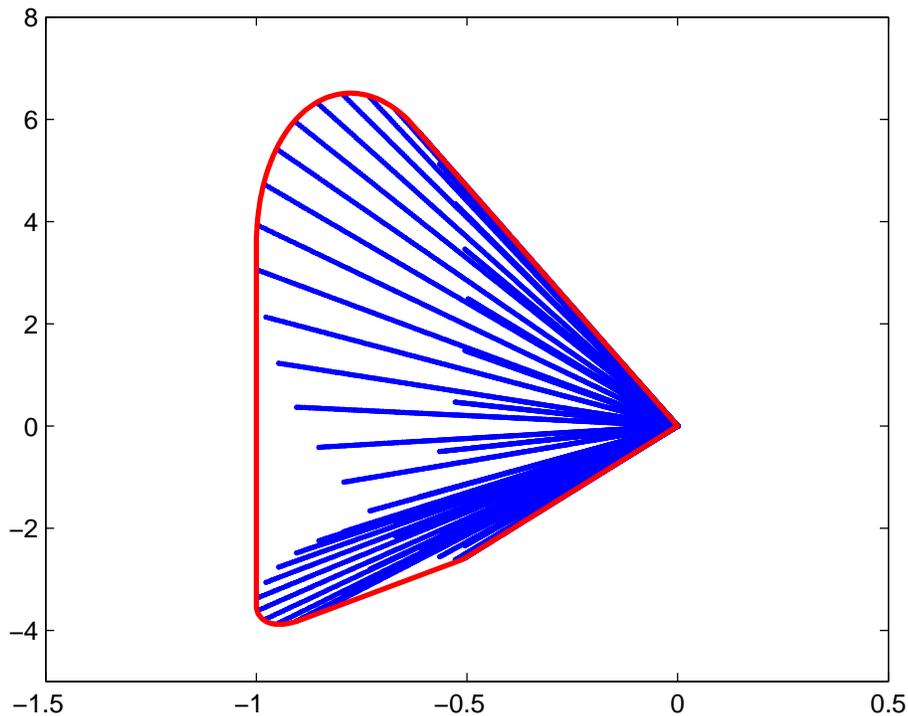}
\etab
\caption{ The dotted area is the set $V$ in Example~\ref{emp:hmg-2qc},
and the outer curve surrounds its convex hull. }
\label{fig-hmg:2q:>&=}
\end{figure}

\begin{exm}  \label{emp:hmg-2qc}
Consider the parameterization
\[
V = \left\{ \bpm 2x_1^2-3x_2^2-4x_3^2 \\  5x_1x_2-7x_1x_3-9x_2x_3 \epm
\left| \baray{c} x_1^2-x_2^2-x_3^2=0, \\  1-x^Tx \geq 0 \earay \right.
\right \}.
\]
The set $V$ is drawn in the dotted area of Figure~\ref{fig-hmg:2q:>&=}.
By Theorem~\ref{hull:2hmg},
the convex hull $\cv{V}$ is given by the following semidefinite representation
\[
\left\{ \bpm 2X_{11}-3X_{22}-4X_{33} \\  5X_{12}-7X_{13}-9X_{23} \epm
\left| \bbm X_{11} & X_{12} & X_{13} \\ X_{12} & X_{22} & X_{23} \\
X_{13} & X_{23} & X_{33} \ebm \succeq 0,
\baray{c} X_{11}-X_{22}-X_{33} =0, \\  1-X_{11}-X_{22}-X_{33}   \geq 0  \earay \right.
\right\}.
\]
The convex region described above
is surrounded by the outer curve in Figure~\ref{fig-hmg:2q:>&=},
which is clearly the convex hull of the dotted area.
\qed

\end{exm}

The conditions like $C_{i,i}$ can not be removed
in Theorem~\ref{hull:2hmg}. We show this
by a counterexample.

\begin{exm}
Consider the quadratically parameterized set
\[
V = \{(x_1x_2, x_1^2):\, 1-x_1x_2 \geq 0, 1+x_2^2-x_1^2 \geq 0\},
\] 
which is motivated by Example~4.4 of \cite{HLNS}.
The condition $C_{i,i}$ is clearly not satisfied.
The semidefinite relaxation $\mathscr{H}_{i,i}$ for $\cv{V}$ is
\[
\{ (X_{12},X_{11}): \, X \succeq 0, 1-X_{12} \geq 0, 1+X_{22}-X_{11} \geq 0 \}.
\]
They are not equal, and neither are their closures.
This is because $V$ is bounded above
in the direction $(1,1)$, while $\mathscr{H}_{i,i}$ is unbounded
(cf. \cite[Example~4.4]{HLNS}).
So, $\overline{ \cv{V} } \ne \overline{ \mathscr{H}_{i,i} }$ for this example,
which is due to the failure of the condition $C_{i,i}$.
\qed
\end{exm}

\section{Rational parametrization}
\setcounter{equation}{0}

Consider the rationally parameterized set
\be \label{var:U}
U \quad = \quad \left\{ \left(
\frac{f_1(x)}{f_0(x)}, \ldots, \frac{f_m(x)}{f_0(x)}
\right): \, x \in T \right\}
\ee
with all $f_0,\ldots, f_m$ being polynomials
and $T$ a semialgebraic set in $\re^n$.
Assume $f_0(x)$ is nonnegative on $T$ and
every $f_i/f_0$ is well defined on $T$,
i.e., the limit $\lim_{x\to z} f_i(x)/f_0(x)$ exists
whenever $f_0$ vanishes at $z\in T$.
The convex hull $\cv{U}$ would be
investigated through considering the polynomial parameterization
\be \label{def:P-hmg}
P \quad = \quad \left\{
\Big(f_1^h(x^h), \ldots, f_m^h(x^h)\Big): \, f_0^h(x^h)=1, \, x^h \in T^h \right\}.
\ee
Here $x^h=(x_0,x_1,\ldots,x_n)$ is an augmentation of $x$ and
\[
f_i^h(x^h) \quad = \quad x_0^{d}f_i(x/x_0) \qquad (d=\max_i \deg(f_i))
\]
is a homogenization of $f_i(x)$,
and $T^h$ is the homogenization of $T$ defined as
\be \label{df:T^h}
T^h = \overline{\{x^h:\, x_0>0, x/x_0 \in T \} }.
\ee
The relation between $\cv{V}$ and $\cv{P}$ is given as below.

\begin{pro} \label{cvx:rat=>poly}
Suppose $f_0(x)$ is nonnegative on $T$ and does not vanish
on a dense subset of $T$, and every $f_i/f_0$ is well defined on $T$. Then
\be \label{rat:clos=}
\overline{ \cv{U} } \quad = \quad \overline{ \cv{P} }.
\ee
Moreover, if $T^h\cap\{f_0^h(x^h)=1\}$ and
$T$ are compact and $f_0(x)$ is positive on $T$, then
\be \label{rt:cvx=}
\cv{U} \quad = \quad  \cv{P}.
\ee
\end{pro}

\begin{proof}
Let $T_1$ be a dense subset of $T$ such that $f_0(x) >0$ for all $x \in T_1$. Clearly,
\[
\overline{ \cv{U} }  = \overline{ \mbox{conv}\left\{
 \left(\frac{f_1^h(x^h)}{f_0^h(x^h)}, \ldots, \frac{f_m^h(x^h)}{f_0^h(x^h)} \right):
x^h \in T_1^h   \right \} }.
\]
Since every $f_i^h$ is homogeneous, we can assume that $f_0^h(x^h)=1$. Then,
\[
\overline{ \cv{U} }  =  \overline{ \mbox{conv}\left\{
\left(f_1^h(x^h), \ldots, f_m^h(x^h) \right): f_0^h(x^h)=1, \,
x^h \in T_1^h   \right \} }.
\]
The density of $T_1$ in $T$ and the above imply \reff{rat:clos=}.

When $T$ is compact and $f_0(x)$ is positive on $T$, $\cv{U}$ is compact.
The $\cv{P}$ is also compact when $T^h\cap\{f_0^h(x^h)=1\}$ is compact.
Thus, \reff{rt:cvx=} follows from \reff{rat:clos=}.
\end{proof}

\noindent
{\it Remark:} If $d=\max_i \deg(f_i)$ is even
and $T$ is defined by polynomials of even degrees,
then we can remove the condition $x_0>0$ in the definition of $T^h$ in \reff{df:T^h}
and Proposition~\ref{cvx:rat=>poly} still holds.

If every $f_i$ in \reff{var:U} is quadratic,
$T$ is defined by a single quadratic inequality,
and $f_0$ is nonnegative on $T$,
then a semidefinite representation for the convex hull $\cv{U}$ or its closure
can be obtained by applying Proposition~\ref{cvx:rat=>poly} and Theorem~\ref{hull:2hmg}.
Suppose $T=\{x: g(x) \geq 0\}$, with $g(x)$ being quadratic.
Write every $f_i^h(x^h) = (x^h)^TF_i\,x^h$ and $g^h(x^h) = (x^h)^TG\,x^h$. Then
\be \label{eq:4.6}
\overline{ \cv{P} } \quad = \quad \overline{ \mbox{conv}\left\{
\Big((x^h)^TF_1\,x^h, \ldots, (x^h)^TF_m\,x^h\Big): \,
\baray{c} (x^h)^TF_0\,x^h=1, \\  x_0>0, (x^h)^TG\,x^h \geq 0 \earay  \right\} }.
\ee
Since the forms $f_i^h$ and $g^h$ are all quadratic,
the condition $x_0>0$ can be removed
from the right hand side of \reff{eq:4.6}, and we get
\be
\overline{ \cv{P} } \quad = \quad \overline{ \mbox{conv}\left\{
\Big((x^h)^TF_1\,x^h, \ldots, (x^h)^TF_m\,x^h\Big): \,
\baray{c} (x^h)^TF_0\,x^h=1, \\ (x^h)^TG\,x^h \geq 0 \earay  \right\} }.
\ee
If there are numbers $\mu_1\in\re$ and $\mu_2\in \re_+$ satisfying
$\mu_1 F_0 + \mu_2 G \prec 0$,
then a semidefinite representation for $\overline{ \cv{P} }$
can be obtained by applying Theorem~\ref{hull:2hmg}.
The case $T=\{x: g(x) = 0\}$ is defined
by a single quadratic equality is similar.

\begin{figure}[htb]
\centering
\btab{c}
\includegraphics[height=.7\textwidth]{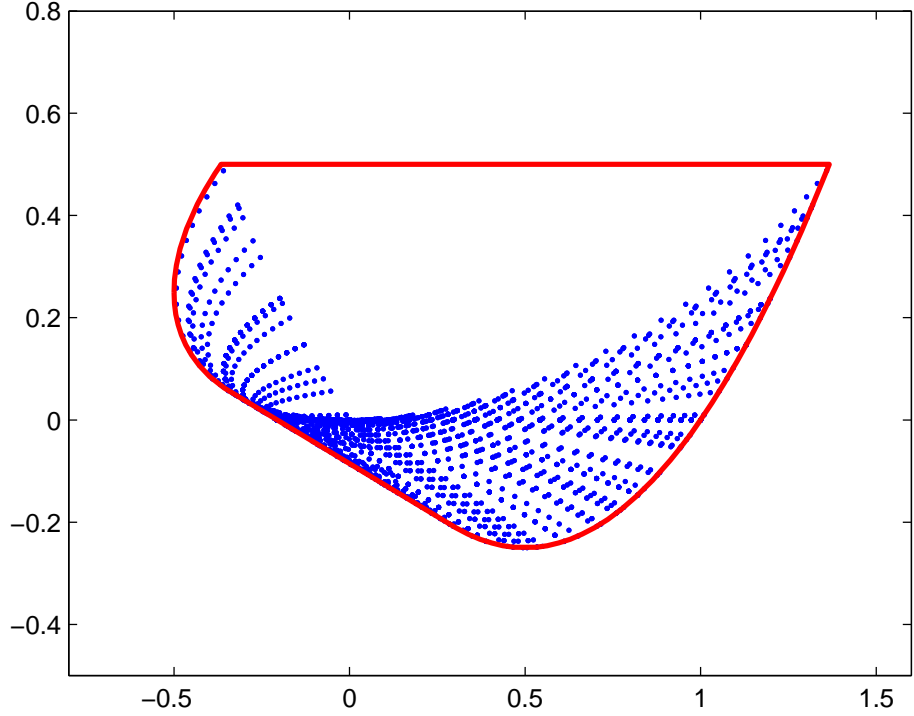}
\etab
\caption{The dotted area is the set $U$ in Example~\ref{emp:ratQ},
and the outer curve is the boundary of its convex hull. }
\label{fig-Qrat:1q}
\end{figure}

\begin{exm} \label{emp:ratQ}
Consider the quadratically rational parametrization:
\[
U = \left\{ \left(\frac{x_1^2+x_2^2+x_3^2+x_1+x_2+x_3}{1+x^Tx},
\frac{x_1x_2+x_1x_3+x_2x_3}{1+x^Tx} \right): \,x_1^2+x_2^2+x_3^2 \leq 1 \right \}.
\]
The dotted area in Figure~\ref{fig-hmg:2q:>&=} is the set $U$ above.
The set $P$ in \reff{def:P-hmg} is
\[
P = \left\{ \bpm x_1^2+x_2^2+x_3^2+x_0(x_1+x_2+x_3) \\
x_1x_2+x_1x_3+x_2x_3 \epm
\left|\baray{c}
x_0^2+x_1^2+x_2^2+x_3^2= 1, \\ x_0^2-x_1^2-x_2^2-x_3^2 \geq 0
\earay \right.
\right \}.
\]
By Theorem~\ref{hull:2hmg},
the convex hull $\cv{P}$ is given by the semidefinite representation
\[
\left\{ \bpm X_{11}+X_{22}+X_{33}+X_{01}+X_{02}+X_{03} \\
X_{12}+X_{13}+X_{23} \epm
\left|\baray{c}
\bbm X_{00} & X_{01} & X_{02} & X_{03} \\
X_{01} & X_{11} & X_{12} & X_{13} \\
X_{02} & X_{12} & X_{22} & X_{23} \\
X_{03} & X_{13} & X_{23} & X_{33} \ebm \succeq 0, \\
X_{00} + X_{11} + X_{22} + X_{33} =1, \\
 X_{00} - X_{11} - X_{22} - X_{33} \geq 0
\earay \right.
\right\}.
\]
The convex region described above is
surrounded by the outer curve in Figure~\ref{fig-Qrat:1q},
which also surrounds the convex hull of the dotted area.
Since $T$ is compact and the denominator $1+x^Tx$ is strictly
positive, $\cv{U} = \cv{P}$ by Proposition~\ref{cvx:rat=>poly}.
\qed

\end{exm}


\begin{thebibliography}{99}


\bibitem{BT}
C. Bayer and J. Teichmann.
The proof of Tchakaloff's Theorem.
{\it Proc. Amer. Math. Soc.}, 134(2006), 3035-3040.

 

\bibitem{FiNi01}
L. Fialkow and J. Nie.
Positivity of Riesz functionals and solutions of
quadratic and quartic moment problems.
{\it J. Functional Analysis},
Vol. 258, No. 1, pp. 328--356, 2010.


 

\bibitem{HLNS}
S. He, Z. Luo, J. Nie and S. Zhang.
Semidefinite Relaxation Bounds for Indefinite Homogeneous Quadratic Optimization.
{\it SIAM Journal on Optimization}, Vol.~19, No.~2, pp. 503--523, 2008.


\bibitem{HN1}
J.W. Helton and J. Nie. Semidefinite representation of convex sets.
{\it Mathematical Programming}, Ser.~A, Vol. 122, No.1, pp.21--64, 2010.


\bibitem{HN2}
J.W. Helton and J. Nie.
Sufficient and necessary conditions for semidefinite representability of convex hulls and sets.
{\it SIAM Journal on Optimization}, Vol. 20, No.2, pp. 759-791, 2009.


\bibitem{HN08}
J.W. Helton and J. Nie.
Structured semidefinite representation of some convex sets.
{\it Proceedings of 47th IEEE Conference on Decision and Control},
pp. 4797 - 4800, Cancun, Mexico, Dec. 9-11, 2008.


\bibitem{HV}
W.~Helton and V.~Vinnikov.
Linear matrix inequality representation of sets.
{\em Comm.~Pure Appl.~Math.}  60 (2007), No. 5, pp. 654-674.

 

\bibitem{Hen09}
D. Henrion.
Semidefinite representation of convex hulls of rational varieties.
LAAS-CNRS Research Report No. 09001, January 2009.

 

\bibitem{Las06}
J.~Lasserre. Convex sets with semidefinite representation.
{\it Mathematical Programming}, Vol.~120, No.~2, pp. 457--477, 2009.


\bibitem{Las08}
J.~Lasserre.
Convexity in semi-algebraic geometry and polynomial optimization.
{\it SIAM Journal on Optimization}, Vol.~19, No.~4, pp. 1995 -- 2014, 2009.


 
 

\bibitem{Nie08}
J.~Nie.
First order conditions for semidefinite representations of convex sets defined by
rational or singular polynomials.
{\it Mathematical Programming}, to appear.

\bibitem{Nie-MOR}
J.~Nie.
Polynomial matrix inequality and semidefinite representation.
{\it Mathematics of Operations Research},
Vol.~36, No.~3, pp. 398--415, 2011.

 

\bibitem{Par06}
P.~Parrilo. Exact semidefinite representation for genus zero
curves. Talk at the Banff workshop ``Positive Polynomials and
Optimization'', Banff, Canada, October 8-12, 2006.


\bibitem{Patk}
G.~Pataki.
On the rank of extreme matrices in semidefinite programs
and the multiplicity of optimal eigenvalues.
{\it Mathematics of Operations Research}, 23 (2), 339--358, 1998.


 

\bibitem{Putinar}
M.~Putinar. Positive polynomials on compact semi-algebraic sets,
{\it  Ind. Univ. Math. J.} \,  42 (1993) 203--206.

 
 
\bibitem{Smg}
K.~Schm\"udgen. The K-moment problem for compact
semialgebraic sets. {\it Math. Ann.} \textbf{289} (1991), 203--206.

 
\end{thebibliography}
\end{document}